\documentclass[oneside,a4paper,12pt,reqno]{amsartwithDOIsupport}

\usepackage{amssymb,amsthm,hyperref}

\def\AA{C_1}
\def\BB{C_2}
\def\CC{C_3}
\def\DD{C_4}
\def\EE{C_5}
\def\FF{C_6}
\def\GG{C_7}
\def\HH{C_8}
\def\II{C_9}
\def\JJ{C_{10}}
\def\KK{C_{11}}
\def\LL{C_{12}}

\renewcommand{\subjclassname}{\textup{2000} Mathematics
Subject Classification}

\newtheorem{theorem}{Theorem}

\newtheorem{lemma}[theorem]{Lemma}

\theoremstyle{definition}

\newtheorem{example}[theorem]{Example}

\def\residue{{\rm{Res}}}
\def\acts{\hspace{-1pt}\mbox{\raisebox{1.3pt}{\text{\huge{.}}}}\hspace{-1pt}}
\def\fix{\mathsf{F}}
\def\orbit{\mathsf{O}}
\def\mertens{\mathcal{M}}
\def\submertens{\mathcal{N}}
\def\eul{{\rm{e}}}
\def\bigo{{\rm{O}}}
\def\littleo{{\rm{o}}}
\def\({\left(}
\def\){\right)}
\def\imag{{\rm{i}}}
\def\dee{{\,\rm{d}}}
\def\lf{{\rm L}_f}
\DeclareMathOperator\girth{girth}
\renewcommand{\le}{\leqslant}
\renewcommand{\ge}{\geqslant}

\begin{document}

\title[A dichotomy in orbit-growth]{A dichotomy in orbit-growth for commuting
automorphisms}
\author{Richard Miles}
\address{School of Mathematics, KTH, SE-100 44,
Stockholm, Sweden}
\email{r.miles@uea.ac.uk}
\author{Thomas Ward}
\address{School of Mathematics, University of
East Anglia, Norwich NR4 7TJ, England}
\email{t.ward@uea.ac.uk}
\thanks{This research was supported by grant KAW 2005.0098 from the Knut and
Alice Wallenberg Foundation and by grant
4806 from the London Mathematical Society.}

\begin{abstract}
We consider asymptotic orbit-counting problems for certain
expansive actions by commuting automorphisms of compact groups. A
dichotomy is found between systems with asymptotically more
periodic
orbits than the topological entropy predicts, and those
for which there is no excess of periodic orbits.
\end{abstract}

\maketitle

Let~$G$ be a countable group acting on some set~$X$, with the
action written~$x\mapsto g\acts x$. Let~$\mathcal L=\mathcal
L(G)$ denote the poset of finite index subgroups of~$G$, and
write~$a_n(G)=\vert\{L\in\mathcal{L}\mid[G:L]=n\}\vert$.
We assume that~$\mathcal L$ is locally finite (a finiteness assumption
on~$G$, guaranteed if~$G$ is finitely generated).
For~$L\in\mathcal{L}$,
the set of~$L$-periodic points in~$X$ under the action
is
\[
\fix(L)=\{x\in X\mid g\acts x=x\mbox{ for all }g\in L\}.
\]
An~$L$-periodic orbit~$\tau$ is the orbit of a point with
stabilizer~$L$, and the length of the orbit is
denoted~$[L]=[G:L]$, the index of~$L$ in~$G$. We always assume that
there are only finitely many orbits of length~$n$ for each~$n\ge1$
(a finiteness assumption on the action, guaranteed
if the action is expansive). The number of~$L$-periodic
orbits is
\[
\orbit(L)=\frac{1}{[L]}\left\vert
\{
x\in X\mid g\acts x=x\Longleftrightarrow g\in L
\}
\right\vert
\]
Orbit growth may be studied via the asymptotic behaviour of the
orbit-counting function
\[
\pi(N)=\sum_{[L]\le N}\orbit(L).
\]
Our focus is on actions with an
exponential rate of orbit growth~$g>0$, and for these it is also natural to consider
the weighted sum
\[
\mertens(N)=\sum_{[L]\le N}\frac{\orbit(L)}{\eul^{g[L]}}.
\]

Any~$L$-periodic point lives on a unique~$L'$-periodic orbit
for some subgroup~$L'\ge L$, so
\begin{equation}\label{hellodarknessmyoldfriend}
\fix(L)=\sum_{L'\ge L}[L']\orbit(L')
\end{equation}
and therefore
\begin{equation}\label{ivecometotalkwithyouagain}
\orbit(L)=\frac{1}{[L]}\sum_{L'\ge L}\mu(L',L)\fix(L'),
\end{equation}
where~$\mu$ is the M{\"o}bius function on the incidence algebra
of~$\mathcal L$ (the equivalence of~\eqref{hellodarknessmyoldfriend}
and~\eqref{ivecometotalkwithyouagain} for all functions~$\fix:\mathcal L\to
\mathbb N_0$ defines the function~$\mu$ by induction).

\begin{example}
The familiar setting for dynamical systems has~$G=\mathbb Z$,
where the action is generated by the
transformation~$x\mapsto 1\acts x$. If there are
parameters~$h>h'>0$ with
\[
\fix(n\mathbb Z)=\eul^{hn}+\bigo(\eul^{h'n}),
\]
then it is easy to check that
\[
\pi(N)\sim\frac{\eul^{h(N+1)}}{N}
\]
and
\[
\mertens(N)=\log N+\AA+\bigo(1/N),
\]
with~$g=h$. Asymptotics of this shape arise in hyperbolic dynamical
systems (see Parry and Pollicott~\cite{MR727704} and
Sharp~\cite{MR1139566}), and in combinatorics (see
Pakapongpun and the second author~\cite{pakapongpun}). Natural examples
with slower growth rates are studied
in~\cite{MR1461206},~\cite{MR2339472},~\cite{MR2180241}.
For example, in~\cite{MR2339472} it is shown that for certain algebraic
dynamical systems of finite combinatorial rank the asymptotic growth
rate takes the form
\[
\pi(N)\sim N^{\sigma}(\log N)^{\kappa}
\]
for some~$\sigma,\kappa\ge0$. In all these cases the growth
comes entirely from the action, because there is no
growth in the group:~$a_n(\mathbb Z)=1$ for all~$n\ge1$,
and~$\vert\mu(L,L')\vert\le 1$ for
all~$L,L'\in\mathcal L(\mathbb Z)$.
\end{example}

\begin{example}\label{becauseavisionsoftlycreeping}
Let~$G$ be a finitely generated nilpotent group
and~$B$ a finite alphabet. The \emph{full~$G$-shift} on~$b=|B|$ symbols
is the~$G$-action on~$B^G$ given by~$(g\acts x)_h=x_{gh}$,
where~$x=(x_h)\in B^G$. For this action
\[
\fix(L)=b^{[L]}
\]
for all~$L\in\mathcal L(G)$ and there is a characterisitic exponential growth of~$\log b$.
We showed
in~\cite{MR2465676} that there are
constants~$\BB>0$,~$\alpha\in\mathbb Q_{\ge0}$
and~$\beta\in\mathbb N_{0}$ for which
\[
\mertens(N)\sim\BB N^{\alpha}\left(\log N\right)^{\beta}.
\]
For~$G=\mathbb Z^d$,~$d\ge2$, there are
constants~$\CC,\DD,\EE>0$ such that
\[
\CC\le\frac{\pi(N)}{N^{d-2}b^{N}}\le\DD(\log N)^{d-1}
\]
and
\[
\mertens(N)\sim\EE N^{d-1}.
\]
In these examples there is exponential growth due to the
action and some growth from the group: in this
setting both~$a_n$ and~$\mu$ are unbounded functions.
\end{example}

In this paper we start to bridge the gap between these
two examples, by considering some actions of~$\mathbb Z^2$
less trivial than the full shift. It is hoped that,
for example, asymptotics for any expansive~$\mathbb Z^d$-action by
automorphisms of a compact group may be found, but the simple
case considered here already throws up new phenomena.

\section{Actions defined by polynomials}

Fix a polynomial~$f\in\mathbb Z[x^{\pm1},y^{\pm1}]$,
written~$f(x,y)=\sum c_{(a,b)}x^ay^b$
for some finitely-supported function~$c:\mathbb Z^2\to\mathbb Z$.
Associate to~$f$
a compact abelian group
\[
X_f=\{x\in\mathbb T^{\mathbb Z^2}\mid\sum c_{(a,b)}x_{(a+m,b+n)}=0\pmod{1}
\mbox{ for all }m,n\in\mathbb Z\},
\]
with the~$\mathbb Z^2$-action defined by the shift,
\[
\left((m,n)\acts x\right)_{(k,\ell)}=x_{(m+k,n+\ell)}.
\]
Assume that~$f(\eul^{2\pi\imag s},\eul^{2\pi\imag t})\neq0$ for
all~$(s,t)\in\mathbb T^2$; by Schmidt~\cite{MR1069512} this is
equivalent to the action being \emph{expansive} with respect to the
natural topology on~$X_f$ inherited from that of~$\mathbb
T^{\mathbb Z^2}=X_{0}$
(that is, there is a neighbourhood~$U$ of~$0\in X_f$ such
that~$\bigcap_{\mathbf(m,n)\in\mathbb{Z}^2}(m,n)\acts U=
\{0\}$).

If~$f(x,y)=b\in\mathbb N$ is a constant, then~$X_f$ is the
full~$\mathbb Z^2$-shift on~$b$ symbols as in
Example~\ref{becauseavisionsoftlycreeping}.
In a wider context, the connection between
algebraic~$G$-actions and polynomials (or ideals)
in the integral group ring~$\mathbb Z[G]$ plays a
central role in algebraic dynamics. An overview
of this theory may be found in Schmidt's
monograph~\cite{MR1345152}, and some recent
developments for groups larger than~$\mathbb Z^d$
include work of the first author~\cite{MR2216558},
of Einsiedler and Rindler~\cite{MR1849144}, and
of Deninger and Schmidt~\cite{MR2322178}.

For brevity we write
\[
\lf(s,t)=\log\vert f(\eul^{2\pi\imag s},\eul^{2\pi\imag t})\vert,
\]
which is a continuous function on~$\mathbb T^2=
\mathbb R^2/\mathbb Z^2$ under
the standing assumption of expansiveness.
Following Lind~\cite{MR1411232}, let~$\mathcal C$ denote the
set of compact subgroups of~$\mathbb T^2$, and define a
map~$\mathsf{m}:\mathcal{C}\to\mathbb R$, continuous in the
Hausdorff metric on~$\mathcal C$, by
\[
\mathsf{m}(K)=
\int_{K}\lf(s,t)\dee m_K(s,t)
\]
where the integration is with respect to Haar measure~$m_K$.
In particular, if~$K$ is a finite subgroup, then~$\mathsf{m}(K)=
\frac{1}{\vert K\vert}\sum_{(s,t)\in K}\lf(s,t).$

By Lind, Schmidt and the second author~\cite{MR1062797},
the topological entropy
of the action is given by
\[
h=\mathsf{m}(\mathbb T^2),
\]
the Mahler measure of~$f$. The growth in periodic points is also studied
in~\cite{MR1062797}, and in particular it is shown that
\[
\lim_{\girth(L)\to\infty}\frac{1}{[L]}\log\fix(L)=h,
\]
where~$\girth(L)=\min\{\Vert(a,b)\Vert\mid(a,b)\in L\setminus\{(0,0)\}\}$. The upper growth
rate is found by Lind~\cite{MR1411232},
\begin{equation}\label{leftitsseedswhileIwassleeping}
g=\lim_{N\to\infty}\sup_{[L]\ge N}\frac{1}{[L]}\log\fix(L)=
\sup_{C\in\mathcal C_{\infty}}\mathsf{m}(K),
\end{equation}
where~$\mathcal C_{\infty}\subset\mathcal C$ is the set of
infinite compact subgroups of~$\mathbb T^2$.

The dichotomy in the title is this: if~$g>h$, then the action
has favoured directions corresponding to sequences of infinite
subgroups along which there is convergence to~$g$
in~\eqref{leftitsseedswhileIwassleeping} (and along which an
abundance of periodic points are found, in excess of the amount
predicted by the topological entropy, which is a global
invariant for the whole action). If~$g=h$ then
there are no preferred directions. Systems with~$g>h$
as a result behave more like the familiar case~$G=\mathbb Z$, while
systems with~$g=h$ have the potential for orbit growth
asymptotics peculiar to higher rank actions (see
Table~\ref{table} for explicit asymptotics for the
simplest~$d$-dimensional systems for small values of~$d$).

Any~$L\in\mathcal L=\mathcal L(\mathbb Z^2)$ may be written in
the form
\begin{equation*}\label{andthevisionplantedinmybrain}
L=L(a,b,c)=\langle(a,0),(b,c)\rangle,
\end{equation*}
where~$a,c\ge1$,~$0\le b\le a-1$, and~$[L]=ac$
(this is the canonical form for lattices due
originally to Hermite~\cite{041.1126cj}).
Write
\[
L^{\perp}=
\{
\left(\textstyle\frac{j}{a},\frac{k}{c}-\frac{jb}{ac}\right)\mid
0\le j\le a-1,0\le k\le c-1
\}
\]
for the annihilator of~$L$ under the
Pontryagin duality
between~$\mathbb T^2$ and~$\mathbb Z^2$.
By~\cite{MR1062797} we have
\begin{equation*}\label{stillremains}
\fix(L)=
\prod_{(s,t)\in L^{\perp}}\vert f(\eul^{2\pi\imag s},\eul^{2\pi\imag t})\vert=
\eul^{[L]\mathsf{m}(L^{\perp})}.
\end{equation*}

\begin{theorem}\label{whenmyeyeswerestabbedbyaflashofneonlight}
If~$g>h$, then there are constants~$\FF,\GG>0$ such that
\begin{equation}\label{thatsplitthenight}
\FF\log N\le\mertens(N)\le\GG\log N
\end{equation}
and
\begin{equation}\label{andtouchedthesoundsofsilence}
\FF\le\frac{\pi(N)}{\eul^{gN}}\le\GG.
\end{equation}
\end{theorem}

\begin{proof}
Just as in~\cite{MR2339472} and~\cite{MR2465676}, part of the proof
involves isolating a main term. However, the more complex
geometry of the acting group and the action requires additional
steps to take account of the preferred directions
with an abundance of periodic orbits.

Associate to~$L(a,b,c)\in\mathcal{L}$ subgroups
\[
J(a)=\{(\textstyle\frac{j}{a},t)\mid t\in\mathbb T,j=0,\dots,a-1\},
\]
\[
J(b,c)=\{(t,\textstyle\frac{k}{c}-\frac{bt}{c})\mid t\in\mathbb T,k=0,\dots,c-1\} \subset\mathcal C_\infty,
\]
and set
\[
K(L)=\begin{cases}
J(a)&\mbox{if }a<c;\\
J(b,c)&\mbox{if }a\ge c.
\end{cases}
\]
The subgroup~$K(L)$ approximates~$L$-periodic points in
the following sense.

\begin{lemma}\label{withinthesoundofsilence}
There is a constant~$\HH$, depending only on~$f$, with
\begin{equation}\label{inrestlessdreamsiwalkedalone}
\left\vert
\mathsf{m}(L^{\perp})-\mathsf{m}(K(L))
\right\vert\le\frac{\HH}{\max\{a,c\}}
\end{equation}
for any~$L=L(a,b,c)\in\mathcal L$.
\end{lemma}

\begin{proof}
For a point~${w}=(w_1,w_2,w_3,w_4)\in\mathbb T^4$,
let~$\Lambda(w)$ denote the line segment from~$(w_1,w_2)$
to~$(w_3,w_4)$, and let~$V(w)$ be the total variation of the
curve~$(s,t)\mapsto\lf(s,t)$ for~$(s,t)\in\Lambda(w)$. By the
hypothesis of expansiveness,~$V:\mathbb T^4\to\mathbb R$ is
continuous and hence bounded by some constant~$\alpha$. Thus
\begin{equation}\label{andinthenakedlightisaw}
\left\vert
\frac{1}{c}\sum_{k=0}^{c-1}
\lf(\textstyle\frac{j}{a},\frac{k}{c}-\frac{bj}{ac})
-
\displaystyle\int_0^1\lf(\textstyle\frac{j}{a},t)\dee t
\right\vert\le\frac{\alpha}{c}
\end{equation}
and
\begin{equation}\label{tenthousandpeoplemaybemore}
\left\vert
\frac{1}{a}\sum_{j=0}^{a-1}
\lf(\textstyle\frac{j}{a},\frac{k}{c}-\frac{bj}{ac})
-\displaystyle\int_0^1
\lf(t,\textstyle\frac{k}{c}-\frac{b}{c}t)\dee t
\right\vert\le\frac{\alpha}{a}.
\end{equation}
If~$a<c$, then
\[
\vert
\mathsf{m}(L^{\perp})-\mathsf{m}(K(L))\vert
\le
\frac{1}{a}
\sum_{j=0}^{a-1}\left\vert
\frac{1}{c}
\sum_{k=0}^{c-1}
\lf(\textstyle\frac{j}{a},
\frac{k}{c}-\frac{bj}{ca})
-
\displaystyle\int_{0}^{1}
\lf(\textstyle\frac{j}{a},t)\dee t
\right\vert\negmedspace,
\]
and so~\eqref{andinthenakedlightisaw}
gives~\eqref{inrestlessdreamsiwalkedalone}. If~$a\ge c$, then
\[
\vert
\mathsf{m}(L^{\perp})-\mathsf{m}(K(L))
\vert
\le
\frac{1}{c}
\sum_{k=0}^{c-1}\left\vert
\frac{1}{a}
\sum_{j=0}^{a-1}
\lf(\textstyle\frac{j}{a},
\frac{k}{c}-\frac{bj}{ca})
-\displaystyle\int_0^1
\lf(t,\textstyle\frac{k}{c}-\frac{b}{c}t)
\dee t\right\vert
\negmedspace\negthinspace,
\]
and in this case~\eqref{tenthousandpeoplemaybemore}
implies~\eqref{inrestlessdreamsiwalkedalone}.
\end{proof}

Write~$\mathcal L(n)$ for the set of subgroups of index~$n$,
and isolate the term corresponding to the largest subgroups
arising in~$\mathcal M(N)$ by writing
\[
\mertens_1(N)=\sum_{n\le N}\frac{1}{n}\sum_{L\in\mathcal L(n)}\frac{\fix(L)}{\eul^{gn}}
\]
and
\[
\mertens_2(N)=\sum_{n\le N}\frac{1}{n}\sum_{L\in\mathcal L(n)}\sum_{L'>L}
\frac{\mu(L',L)\fix(L')}{\eul^{gn}},
\]
so that~$\mertens(N)=\mertens_1(N)+\mertens_2(N)$.

Now fix a subgroup~$L\in\mathcal L(n)$ and assume
that~$L'=L(a,b,c)>L$. Then~$[L']=ac\le\frac{n}{2}$, so
either~$a\le\sqrt{n/2}$ or~$c\le\sqrt{n/2}$.
By~\eqref{inrestlessdreamsiwalkedalone},
\[
[L']\left(
\mathsf{m}(L'^{\perp})-\mathsf{m}(K(L'))\right)
\le\frac{\HH ac}{\max\{a,c\}}=\HH\min\{a,c\}\le\HH\sqrt{n}.
\]
It follows that
\begin{eqnarray*}
\log\fix(L')-gn&=&[L']\mathsf{m}(L'^{\perp})-gn\\
&=&[L']\mathsf{m}(K(L'))-{gn}+[L']\left(
\mathsf{m}(L'^{\perp})-\mathsf{m}(K(L'))
\right)\\
&\le&\displaystyle\frac{n}{2}\left(
\mathsf{m}(K(L'))-g
\right)+\HH\sqrt{n}-\frac{gn}{2}\\
&\le&\HH\sqrt{n}-\displaystyle\frac{gn}{2},
\end{eqnarray*}
since~$\mathsf{m}(K(L'))\le g$
by~\eqref{leftitsseedswhileIwassleeping}.
By~\cite[Lem.~2]{MR2465676} there is a constant~$\II$ with
\[
\vert\mu(L',L)\vert\le\eul^{\II(n/2)^2}
\]
(since~$[L']<n/2$); moreover
\[
a_n(\mathbb Z^2)\le9n\log n
\]
by~\cite[Lem.~3]{MR2465676}. Thus
\begin{eqnarray}
\vert\mertens_2(N)\vert&\le&
\sum_{n\le N}\displaystyle\frac{1}{n}\exp\left({\HH\sqrt{n}-gn/2}\sum_{L\in\mathcal{L}(n)}\right)
\sum_{L'>L}\vert\mu(L',L)\vert\nonumber\\
&\le&
\sum_{n\le N}\displaystyle\frac{9}{n}
\exp\left({\HH\sqrt{n}+\II(\log(n/2))^2-gn/2}\right)n\log n\nonumber\\
&=&\bigo(1).\nonumber\label{narrowstreetsofcobblestone}
\end{eqnarray}
It follows that the asymptotic growth is controlled by~$\mertens_1(N)$.
In order to isolate the subgroups responsible for the excess of
periodic orbits above the level predicted by the topological
entropy, let
\[
A=\{a\ge1\mid\mathsf{m}(J(a))=g\}
\]
and
\[
B=\{(b,c)\mid c\ge1,0\le b\le a-1,\mathsf{m}(J(b,c))=g\}.
\]
Partition the subgroups~$\mathcal{L}(n)$ into
\begin{eqnarray*}
\mathcal{L}_1(n)&=&\{L(a,b,c)\in\mathcal{L}(n)\mid a\in A,(b,c)\in B\},\\
\mathcal{L}_2(n)&=&\{L(a,b,c)\in\mathcal{L}(n)\mid a\in A,(b,c)\notin B\},\\
\mathcal{L}_3(n)&=&\{L(a,b,c)\in\mathcal{L}(n)\mid a\notin A,(b,c)\in B\},\mbox{ and}\\
\mathcal{L}_4(n)&=&\{L(a,b,c)\in\mathcal{L}(n)\mid a\notin A,(b,c)\notin B\}.
\end{eqnarray*}
The sumpremum in~\eqref{leftitsseedswhileIwassleeping} is attained,
so~$\mathcal{L}_2(n)\cup\mathcal{L}_3(n)\neq\emptyset$
for any~$n\ge1$.
This main term decomposes as
\begin{equation}\label{allthestreetsarecrammedwiththings}
\mertens_1(N)=\sum_{j=1}^{4}\underbrace{\sum_{n\le N}\displaystyle\frac{1}{n}
\sum_{L\in\mathcal{L}_j(n)}\exp(n(\mathsf{m}(L^{\perp})-g))}_{\submertens_j(N)}.
\end{equation}
Let~$\mathcal{K}=\{K(L)\mid
L\in\mathcal{L}\}\subset\mathcal{C}$, and
enumerate~$\mathcal{K}=\{K_1,K_2,\dots\}$. In the Hausdorff
metric,~$K_j\rightarrow\mathbb T^2$ as~$j\to\infty$,
so~$\mathsf{m}(K_j)\rightarrow h$ as~$j\to\infty$. Since we
have~$g>h$, it follows that
\begin{equation}\label{neaththehaloofastreetlamp}
\lambda=\inf\{\vert g-\mathsf{m}(K)\vert\mid K\in\mathcal{K},
\mathsf{m}(K)\neq g\}>0
\end{equation}
and
\[
\{J(a)\mid a\in A\}\cup\{ J(b,c)\mid(b,c)\in B\}=
\{K\in\mathcal{K}\mid\mathsf(K)=g\}
\]
must be finite. In particular, both~$A$ and~$B$ are finite,
so~$\mathcal{L}_1(n)\neq\emptyset$ for only finitely many
values of~$n\ge1$, and therefore~$\submertens_1(N)=\bigo(1)$.

If~$L=L(a,b,c)\in\mathcal{L}_4(N)$, then~$a\le\sqrt{n}$
or~$c\le\sqrt{n}$ since~$[L]=ac=n$.
By~\eqref{inrestlessdreamsiwalkedalone}
and~\eqref{neaththehaloofastreetlamp}, it follows that
\begin{eqnarray*}
n\left(
\mathsf{m}(L^{\perp})-g
\right)&=&
n\left(
\mathsf{m}(K(L))-g
\right)+
n\left(
\mathsf{m}(L^{\perp})-\mathsf{m}(K(L))\right)\\
&\le&-\lambda n+\HH\sqrt{n}
\end{eqnarray*}
and so
\begin{equation*}\label{iturnedmycollartothecoldanddamp}
\submertens_4(N)\le\sum_{n\le N}\displaystyle\frac{1}{n}
\vert\mathcal{L}_4(N)\vert\exp(-\lambda n+\HH\sqrt{n}).
\end{equation*}
By~\cite[Lem.~3]{MR2465676},
\[
\vert\mathcal{L}_4(n)\vert\le\vert\mathcal{L}(n)\vert\le 9n\log n,
\]
so~$\submertens_4(N)=\bigo(1).$

We are left with~$\submertens_2$ and~$\submertens_3$. Let
\[
B_a(x)=\left\{
(b,c)\in\mathbb Z^2\mid
\max_{a\in A}\{a\}<c\le\lfloor x\rfloor,
0\le b\le a-1,(b,c)\notin B
\right\}\negmedspace,
\]
so that $\submertens_2(N)=\Theta(N)+\bigo(1)$, where
\[
\Theta(N)=\sum_{a\in A}\sum_{(b,c)\in B_a(N/a)}\displaystyle\frac{1}{ac}
\exp\left(
ac(\mathsf{m}(L(a,b,c)^{\perp})-g
\right).
\]
If~$a\in A$,~$(b,c)\in B_a(N/a)$, and~$L=L(a,b,c)$,
then~$K(L)=J(a)$, so
\begin{eqnarray*}
\vert ac(\mathsf{m}(L^{\perp})-g)\vert&=&
\vert ac(\mathsf{m}(K(L))-g)+ac(\mathsf{m}(L^{\perp})-\mathsf{m}(K(L)))\vert\\
&=&\vert ac(\mathsf{m}(L^{\perp})-\mathsf{m}(K(L)))\vert\le\HH a
\end{eqnarray*}
by~\eqref{inrestlessdreamsiwalkedalone}. Thus
\[
\sum_{a\in A}\displaystyle\frac{1}{a}\exp(-\HH a)
\negmedspace\negmedspace\sum_{(b,c)\in B_a(N/a)}\negmedspace\negmedspace\frac{1}{c}
\le
\Theta(N)
\le
\sum_{a\in A}\displaystyle\frac{1}{a}\exp(\HH a)
\negmedspace\negmedspace\sum_{(b,c)\in B_a(N/a)}\negmedspace\negmedspace\frac{1}{c}.
\]
Now
\[
\sum_{(b,c)\in B_a(N/a)}\frac{1}{c}=a\log\lfloor N/a\rfloor+\bigo(1)=a\log N+\bigo(1),
\]
which when summed over the finitely many possible~$a$ gives the
contribution from~$\submertens_2(N)$.

Now let
\[
A(x)=\left\{
a\in\mathbb Z\mid\max_{(b,c)\in B}\{c\}<a\le\lfloor x\rfloor,
a\notin A
\right\},
\]
so that~$\submertens_3(N)=\Phi(N)+\bigo(1)$, where
\[
\Phi(N)=\sum_{(b,c)\in B}\sum_{a\in A(N/c)}\displaystyle\frac{1}{ac}
\exp(ac(\mathsf{m}(L(a,b,c)^{\perp})-g)).
\]
If~$(b,c)\in B$,~$a\in A(N/c)$, and~$L=L(a,b,c)$,
then~$K(L)=J(b,c)$, so~\eqref{inrestlessdreamsiwalkedalone}
says that~$\vert ac(\mathsf{m}(L^{\perp})-g)\vert\le\HH c$ and
hence
\[
\sum_{(b,c)\in B}\displaystyle\frac{1}{c}\exp(-\HH c)
\sum_{a\in A(N/c)}\frac{1}{a}
\le
\Phi(N)
\le
\sum_{(b,c)\in B}\displaystyle\frac{1}{c}\exp(\HH c)
\sum_{a\in A(N/c)}\frac{1}{a}.
\]
Once again the Euler formula for~$\sum_{a\in
A(N/c)}\frac{1}{a}$ gives the contribution
from~$\submertens_3(N)$.

Finally, we need to check that the constants associated
with~$\submertens_2(N)$ and~$\submertens_3(N)$ cannot both
vanish. This follows from the fact
that
\[
\mathcal{L}_2(n)\cup\mathcal{L}_3(n)\neq\emptyset,
\]
which in turn is a consequence of the fact that the supremum
in~\eqref{leftitsseedswhileIwassleeping} is attained
by~\cite{MR1411232}, completing the proof
of~\eqref{thatsplitthenight}.

Turning to~\eqref{andtouchedthesoundsofsilence}, we isolate a
dominant term as before,
\[
\pi(N)=\sum_{n\le N}\sum_{L\in\mathcal L(n)}\orbit(L)
=
\underbrace{\sum_{n\le N}\frac{1}{n}
\sum_{L\in\mathcal L(n)}\fix(L)}_{\pi_1(N)}
\]
\[
\qquad\qquad\qquad\qquad+ \underbrace{\sum_{n\le N}\frac{1}{n}\sum_{L\in\mathcal
L(n)}\sum_{L'>L}\mu(L',L)\fix(L')}_{\pi_2(N)}.
\]
Then (using estimates from~\cite[Lem.~2,3]{MR2465676} and
Lemma~\ref{withinthesoundofsilence} as before)
\begin{eqnarray*}
\frac{\pi_2(N)}{\eul^{gN}}
&\le&\sum_{n\le N}{\textstyle\frac{1}{n}}\exp(-gn)\sum_{L\in\mathcal L(n)}
\sum_{L'>L}\left(
\exp(\II(\log[L])^2)
\right)\fix(L')\\
&\le&
\sum_{n\le N}{\textstyle\frac{1}{n}}\sum_{L\in\mathcal L(n)}
\sum_{L'>L}
\exp\left(
\II(\log n)^2\vphantom{{\textstyle\frac{\sqrt{n}}{\sqrt{2}}}}\right.\\
&&\qquad\qquad\qquad\qquad\qquad+\left.{\textstyle\frac{n}{2}}\mathsf{m}(K(L'))+\HH
{\textstyle\frac{\sqrt{n}}{\sqrt{2}}}-gN
\right)\\
&\le&
\sum_{n\le N}{\textstyle\frac{1}{n}}\sum_{L\in\mathcal L(n)}
\sum_{L'>L}
\exp\left({\vphantom{\sum}}\right.
\II(\log n)^2+{\textstyle\frac{n}{2}}\underbrace{(\mathsf{m}(K(L'))-g)}_{\le0}\\
&&\qquad\qquad\qquad\qquad\qquad\qquad\qquad
+
\HH{\textstyle\frac{\sqrt{n}}{\sqrt{2}}}-g\underbrace{(N-{\textstyle\frac{n}{2}})}_{\ge N/2}
\left.{\vphantom{\sum}}\right)\\
&\le&
\sum_{n\le N}
{\textstyle\frac{9n^4\log n}{n}}\exp
\left(
\II(\log n)^2+\HH{\textstyle\frac{\sqrt{n}}{\sqrt{2}}}-g{\textstyle\frac{N}{2}}
\right)=\bigo(1).
\end{eqnarray*}
We decompose the main term~$\pi_1(N)$
as~$\sum_{j=1}^{4}\rho_j(N)$, corresponding to the
decomposition~$\mathcal L(n)=\mathcal{L}_1(n)\sqcup
\mathcal{L}_2(n)\sqcup\mathcal{L}_3(n)\sqcup\mathcal{L}_4(n)$
as before.

Since~$A$ and~$B$ are finite, it is easy to check
that~$\exp(-gN)\rho_1(N)\rightarrow0$
as~$N\to\infty$.

If~$L\in\mathcal L_4(n)$ then~$a\le\sqrt{n}$ or~$c\le\sqrt{n}$,
so
\begin{eqnarray*}
n\mathsf{m}(L^{\perp})-Ng&=&
n\underbrace{(\mathsf{m}(K(L))-g)}_{\le-\lambda}+
n\underbrace{(\mathsf{m}(L^{\perp})-\mathsf{m}(K(L)))}_{\le\HH\sqrt{n}}\\
&&\qquad\qquad\qquad\qquad-g(N-n),
\end{eqnarray*}
and therefore
\[
\frac{\rho_4(N)}{\eul^{gN}}
=
\sum_{n\le N}\frac{1}{n}\sum_{L\in\mathcal L_4(N)}\exp\left(
n\mathsf{m}(L^{\perp})-Ng
\right)\rightarrow0
\]
as~$N\to\infty$.

Now
\begin{eqnarray*}
\frac{\rho_2(N)}{\eul^{gN}}&=&\exp(-gN)
\sum
{\textstyle\frac{1}{ac}} \exp(
ac(\mathsf{m}(L(a,b,c)^{\perp})-g)
)\exp(acg)\\
&=&\overbrace{
\exp(-gN)
\sum 1}^{\Upsilon(N)}
+\medspace\littleo(1),
\end{eqnarray*}
where the sum runs over all positive integers $a,b,c$ such that $ac\le N$, $b\le a-1$, $a\in A$ and $(b,c)\notin B$. 
If~$a\in A$ and~$(b,c)\in B_a(N/a)$, then~$K(L(a,b,c))=J(a)$,
so~$\mathsf{m}(K(L))=g$ and
\[
\vert
ac
\(\mathsf{m}(L(a,b,c)^{\perp})-g\)
\vert\le\HH a.
\]
Thus~$\Upsilon(N)$ lies between
\[
\exp(-gN)\sum_{a\in A}\frac{\exp(-\HH a)}{a}
\sum_{(b,c)\in B_a(N/a)}\frac{1}{c}\exp(acg)
\]
and
\[
\exp(gN)\sum_{a\in A}\frac{\exp(-\HH a)}{a} \sum_{(b,c)\in
B_a(N/a)}\frac{1}{c}\exp(acg).
\]
A similar argument applies to~$\rho_3(N)$. This gives the
lower bound in~\eqref{andtouchedthesoundsofsilence} by
considering a single value of~$a$, and the upper bound by easy
estimates.
\end{proof}

\section{Examples}

Theorem~\ref{whenmyeyeswerestabbedbyaflashofneonlight} is a
weak result -- in that it does not give a single asymptotic --
and it only applies when~$g$ exceeds~$h$. This section 
provides exact asymptotics for examples in both the cases~$g>h$ and~$g=h$,
and shows that   
there are actions defined by non-constant polynomials that 
behave more like the~$\mathbb Z^2$-actions in~\cite{MR2465676}.
That is, there are examples beyond full-shifts for which~$\mertens(N)$ behaves like~$N$ rather than~$\log N$. 

\begin{example}
Let~$f(x,y)=2+xy^2$, so that~$h=\log2$
(see~\cite{MR1700272}; this and all subsequent integrations may be computed
using Jensen's formula). Moreover,
\begin{eqnarray*}
\mathsf{m}(J(a))&=&\frac{1}{a}\sum_{j=0}^{a-1}
\int_{0}^{1}\log\vert
2+\eul^{2\pi\imag j/a}\eul^{4\pi\imag t}
\vert\dee t\\
&=&\log2.
\end{eqnarray*}
We calculate~$\mathsf{m}(J(b,c))$
by exploiting the periodicity of~$(s,t)\mapsto\lf(s,t)$:
\[
\mathsf{m}(J(b,c))=\frac{1}{\gcd(b,c)}\sum_{\ell=0}^{\gcd(b,c)-1}
\int_0^1\log\vert
2+\xi_{\ell,c}\eul^{2\pi\imag(c-2b)t}
\vert\dee t,
\]
where~$\xi_{\ell,c}=\eul^{4\pi\imag\ell/c}$. If~$c\neq 2b$
then~$\mathsf{m}(J(b,c))=\log 2$. If~$c=2b$ then the
integrand is~$\log\vert2+\xi_{\ell,c}\vert$, so
\begin{eqnarray*}
\mathsf{m}(J(b,c))&=&\frac{1}{b}\sum_{\ell=0}^{b-1}
\log\vert2+\xi_{\ell,2b}\vert\\
&=&\frac{1}{b}\log\prod_{\ell=0}^{b-1}\vert2+\xi_{\ell,2b}\vert\\
&=&\frac{1}{b}\log(2^b-(-1)^b),
\end{eqnarray*}
which is~$\log3$ when~$c=2b=2$ and is strictly smaller
than~$\log3$ otherwise. Therefore,~$g=\log3$. Following the proof of 
Theorem~\ref{whenmyeyeswerestabbedbyaflashofneonlight}, the
significant contribution to~$\mertens_1(N)$ comes 
from~$\submertens_3(N)$ in~\eqref{allthestreetsarecrammedwiththings}.
Now
\begin{eqnarray*}
\fix(L(a,1,2))&=&\prod_{(s,t)\in L(a,1,2)^{\perp}}\vert
f(\eul^{2\pi\imag s},\eul^{2\pi\imag t})
\vert\\
&=&\prod_{j=0}^{a-1}\prod_{k=0}^{1}\vert
2+\eul^{2\pi\imag j/a}\eul^{4\pi\imag(k/2-j/2a)}
\vert=3^{2a}.
\end{eqnarray*}
Thus
\[
\mertens(N)=\submertens_3(N)+\bigo(1)=\textstyle\frac{1}{2}\log N+\bigo(1).
\]
\end{example}

\begin{example}
Let~$f(x,y)=3+x+y$, so that~$h=\log3$. From
Lind~\cite{MR1411232} we
have~$\mathsf{m}(J(0,1))=\log4$,~$\mathsf{m}(J(1))=\log4$,~$\mathsf{m}(K(L(a,b,c)))<\log4$
for~$(b,c)\neq(0,1)$, and~$\mathsf{m}(K(L(a,b,c)))<\log4$
for~$a\neq1$. Thus we must take both~$\submertens_2(N)$
and~$\submertens_3(N)$ into account. A calculation using
circulants shows that
\[
\fix(L(a,0,1))=4^a-(-1)^a
\]
and
\[
\fix(L(1,b,c))=4^c-(-1)^c,
\]
so
\[
\submertens_2(N)=\sum_{c=2}^{N}\textstyle\frac{1}{c}\left(1-(-4)^{-c}\right)=\log N+\bigo(1)
\]
and
\[
\submertens_3(N)=\sum_{a=2}^{N}\textstyle\frac{1}{a}\left(1-(-4)^{-a}\right)=\log N+\bigo(1).
\]
As in the proof of
Theorem~\ref{whenmyeyeswerestabbedbyaflashofneonlight}, all
other contributions are bounded, so
\[
\mertens(N)=2\log N+\bigo(1).
\]
\end{example}

\begin{example}
Consider the~$d$-dimensional full shift on~$b$ symbols,
which has~$h=g=\log b$ (for~$d=2$ this is the case corresponding to
the polynomial~$f=b$). Then the estimates
from~\cite{MR2465676} show that the growth in~$\mertens(N)$ is
determined by the main term
\[
\sum_{n\le N}\frac{1}{b^{n}}\frac{1}{n}\sum_{L\in\mathcal
L(n)}\fix(L),
\]
and~$\fix(L)=b^{[L]}$. Then
\begin{eqnarray*}
\sum_{n\le N}\frac{b^{-n}\frac{1}{n}a_n(\mathbb Z^d)\eul^{gn}}{n^z}
&=&\sum_{n\ge1}\frac{a_n(\mathbb Z^d)}{n^{z+1}}\\
&=&\zeta(z+1)\zeta(z)\cdots\zeta(z-d+2),
\end{eqnarray*}
so by Perron's theorem~\cite{MR0185094} we have
\begin{eqnarray*}
\mertens(N)&\sim&\residue_{z=d-1}\(
\textstyle\frac{\zeta(z+1)\cdots\zeta(z-d+2)N^z}{z}\)\\
&=&N^{d-1}
\frac{\pi^{\lfloor\frac{d}{2}\rfloor(\lfloor\frac{d}{2}\rfloor+1)}}{r_d}
\negmedspace\prod_{j=1}^{\lfloor(d-1)/2\rfloor}
\negmedspace\zeta(2j+1)
\end{eqnarray*}
for some~$r_d\in\mathbb Q$ ($r_d\in\mathbb N$ for~$d\le11$;
the numerator and denominator of~$r_d$ as~$d$ varies
are sequences~\href{http://www.research.att.com/~njas/sequences/A159283}{A159283}
and~\href{http://www.research.att.com/~njas/sequences/A159282}{A159282} in the
\href{http://www.research.att.com/~njas/sequences/}{On-line Encyclopedia of Integer Sequences}).
This gives the main term in the
dynamical Mertens' theorem for the full~$\mathbb Z^d$-shift
considered in~\cite{MR2465676} in closed form; the first few
expressions are shown in Table~\ref{table}.
\begin{table}[ht]
\begin{center}
\caption{\label{table}}
\begin{tabular}{l|c}
$d$&$\mertens(N)$ for the full $\mathbb Z^d$-shift\\
\hline
$1$&$\log N+\gamma$\\
$2$&$\textstyle\frac{\vphantom{A^A}1}{6}\pi^2N$\\
$3$&$\textstyle\frac{\vphantom{A^A}1}{12}\zeta(3)\pi^2N^2$\\
$4$&$\textstyle\frac{\vphantom{A^A}1}{1620}\zeta(3)\pi^6N^3$\\
$5$&$\textstyle\frac{\vphantom{A^A}1}{2160}\zeta(3)\zeta(5)\pi^6N^4$\\
$6$&$\textstyle\frac{\vphantom{A^A}1}{2551500}\zeta(3)\zeta(5)\pi^{12}N^5$\\
$7$&$\textstyle\frac{\vphantom{A^A}1}{3061800}\zeta(3)\zeta(5)\zeta(7)\pi^{12}N^6$\\
$8$&$\textstyle\frac{\vphantom{A^A}1}{33756345000}\zeta(3)\zeta(5)\zeta(7)\pi^{20}N^7$\\
\end{tabular}
\end{center}
\end{table}
The authors admit that this closed form was
overlooked in~\cite{MR2465676}.
\end{example}

\begin{example}
A simple example beyond the full shift but still with~$g=h$ is
given by~$f(x,y)=x-2$. Here~$h=\log2$,
\begin{eqnarray*}
\mathsf{m}(J(a))&=&\frac{1}{a}\sum_{j=0}^{a-1}
\int_0^1\log\vert\eul^{2\pi\imag j/a}-2\vert\dee t\\
&=&\frac{1}{a}\sum_{j=0}^{a-1}\log\vert
\eul^{2\pi\imag j/a}-2\vert\\
&=&\frac{1}{a}\log(2^a-1),
\end{eqnarray*}
and
\begin{eqnarray*}
\mathsf{m}(J(b,c))=\frac{1}{c}
\sum_{k=0}^{c-1}\int_0^1
\log\vert\eul^{2\pi\imag t}-2\vert\dee t=\log2,
\end{eqnarray*}
so~$g=\log 2$. Now
\[
\fix(L(a,b,c))=\prod_{j=0}^{a-1}\prod_{k=0}^{c-1}\vert
\eul^{2\pi\imag j/a}-2
\vert=(2^a-1)^c,
\]
so~$\eul^{-gac}\fix(L(a,b,c))\le1$ and
\begin{eqnarray*}
\mertens_1(N)&\le&
\sum_{\genfrac{}{}{0pt}{}{a,b,c\ge1,}{0\le b\le a-1; ac\le N}}\frac{1}{ac}\\
&=&
\sum_{c=1}^{N}\frac{1}{c}\displaystyle\sum_{a=1}^{\lfloor N/c\rfloor}1\\
&=&\sum_{c=1}^{N}{\frac{1}{c}}\left(N/c+\bigo(1)\right)\\
&\le&\JJ N
\end{eqnarray*}
for some constant~$\JJ>0$. On the other hand, if~$2^a\ge N$
then
\[
1-2^{-a}\ge1-1/N
\]
so
\[
\underbrace{\exp(-gac)}_{1/2^{ac}}\fix(L(a,b,c))=(1-2^{-a})^c\ge
(1-1/N)^N\ge\textstyle\frac{1}{4}
\]
for~$N\ge2$. It follows that
\begin{eqnarray*}
\mertens_1(N)&\ge&\frac{1}{4}\sum_{c=1}^{N}\frac{1}{c}\sum_{a=\lceil\log_2N\rceil}^{\lfloor N/c\rfloor}1\\
&\ge&\frac{1}{4}
\sum_{c=1}^{\lfloor N/2\log_2 N\rfloor}\frac{1}{c}\left(
\lfloor N/c\rfloor
-\lceil\log_2n\rceil
\right)\ge\KK N
\end{eqnarray*}
for some constant~$\KK>0$ and all sufficiently large~$N$. Thus
\[
0<\KK N\le\mertens(N)\le\LL N
\]
for all large~$N$.
\end{example}


\begin{thebibliography}{10}

\bibitem{MR1461206} (1461206)
V.~Chothi, G.~Everest, and T.~Ward,  `{$S$}-integer dynamical systems: periodic
  points', \emph{J. Reine Angew. Math.} \textbf{489} (1997), 99--132.

\bibitem{MR2322178} (2322178)
C.~Deninger and K.~Schmidt,  `Expansive algebraic actions of discrete
  residually finite amenable groups and their entropy', \emph{Ergodic Theory
  Dynam. Systems} \textbf{27} (2007), no.~3, 769--786.

\bibitem{MR1849144} (1849144)
M.~Einsiedler and H.~Rindler,  `Algebraic actions of the discrete {H}eisenberg
  group and other non-abelian groups', \emph{Aequationes Math.} \textbf{62}
  (2001), no.~1-2, 117--135.

\bibitem{MR2339472} (2339472)
G.~Everest, R.~Miles, S.~Stevens, and T.~Ward,  `Orbit-counting in
  non-hyperbolic dynamical systems', \emph{J. Reine Angew. Math.} \textbf{608}
  (2007), 155--182.

\bibitem{MR2180241} (2180241)
G.~Everest, V.~Stangoe, and T.~Ward,  `Orbit counting with an isometric
  direction', in \emph{Algebraic and topological dynamics}, in \emph{Contemp.
  Math.} \textbf{385}, pp.~293--302 (Amer. Math. Soc., Providence, RI, 2005).

\bibitem{MR1700272} (1700272)
G.~Everest and T.~Ward, \emph{Heights of polynomials and entropy in algebraic
  dynamics}, in \emph{Universitext} (Springer-Verlag London Ltd., London,
  1999).

\bibitem{MR0185094} (0185094)
G.~H. Hardy and M.~Riesz, \emph{The general theory of {D}irichlet's series}, in
  \emph{Cambridge Tracts in Mathematics and Mathematical Physics, No. 18}
  (Stechert-Hafner, Inc., New York, 1964).

\bibitem{041.1126cj}
C.~Hermite,  `{Sur l'introduction des variables continues dans la th\'eorie des
  nombres.}', \emph{J. Reine Angew. Math.} \textbf{41} (1851), 191--216.

\bibitem{MR1062797} (1062797)
D.~Lind, K.~Schmidt, and T.~Ward,  `Mahler measure and entropy for commuting
  automorphisms of compact groups', \emph{Invent. Math.} \textbf{101} (1990),
  no.~3, 593--629.

\bibitem{MR1411232} (1411232)
D.~A. Lind,  `A zeta function for {${\mathbb Z}\sp d$}-actions', in
  \emph{Ergodic theory of ${\mathbb Z}\sp d$ actions (Warwick, 1993--1994)}, in
  \emph{London Math. Soc. Lecture Note Ser.} \textbf{228}, pp.~433--450
  (Cambridge Univ. Press, Cambridge, 1996).

\bibitem{MR2216558} (2216558)
R.~Miles,  `Expansive algebraic actions of countable abelian groups',
  \emph{Monatsh. Math.} \textbf{147} (2006), no.~2, 155--164.

\bibitem{MR2465676} (2465676)
R.~Miles and T.~Ward,  `Orbit-counting for nilpotent group shifts', \emph{Proc.
  Amer. Math. Soc.} \textbf{137} (2009), no.~4, 1499--1507.

\bibitem{pakapongpun}
A.~Pakapongpun and T.~Ward,  `Functorial orbit counting', \emph{J. Integer
  Sequences} \textbf{12} (2009), Article 09.2.4.

\bibitem{MR727704} (727704)
W.~Parry and M.~Pollicott,  `An analogue of the prime number theorem for closed
  orbits of {A}xiom {A} flows', \emph{Ann. of Math. (2)} \textbf{118} (1983),
  no.~3, 573--591.

\bibitem{MR1069512} (1069512)
K.~Schmidt,  `Automorphisms of compact abelian groups and affine varieties',
  \emph{Proc. London Math. Soc. (3)} \textbf{61} (1990), no.~3, 480--496.

\bibitem{MR1345152} (1345152)
K.~Schmidt, \emph{Dynamical systems of algebraic origin}, in \emph{Progress in
  Mathematics} \textbf{128} (Birkh\"auser Verlag, Basel, 1995).

\bibitem{MR1139566} (1139566)
R.~Sharp,  `An analogue of {M}ertens' theorem for closed orbits of {A}xiom {A}
  flows', \emph{Bol. Soc. Brasil. Mat. (N.S.)} \textbf{21} (1991), no.~2,
  205--229.

\end{thebibliography}

\providecommand{\bysame}{\leavevmode\hbox to3em{\hrulefill}\thinspace}

\end{document}